\documentclass{article}

\usepackage{amsmath}
\usepackage{amsthm}
\usepackage{amsfonts}
\usepackage{amssymb}
\usepackage{hyperref}

\newtheorem{thm}{Theorem}[section]
\newtheorem{lem}[thm]{Lemma}

\theoremstyle{definition}

\theoremstyle{remark}

\newcommand{\sbu}{\sigma^{**}}
\newcommand{\floor}[1]{\left\lfloor#1\right\rfloor}

\author{Tomohiro Yamada}
\title{$2$ and $9$ are the only biunitary superperfect numbers\footnote{2010 Mathematics 
Subject Classification: 11A05, 11A25.}
\footnote{Key words: Odd perfect numbers, biunitary superperfect numbers, unitary divisors; biunitary divisors; the sum of divisors.}}
\date{}

\begin{document}
\maketitle

\begin{abstract}
We shall show that $2$ and $9$ are the only biunitary superperfect numbers.
\end{abstract}

\section{Introduction}\label{intro}
As usual, we let $\sigma(N)$ and $\omega(N)$ denote respectively
the sum of divisors and the number of distinct prime factors of a positive integer $N$.
$N$ is called to be perfect if $\sigma(N)=2N$.
It is a well-known unsolved problem whether or not
an odd perfect number exists.  Interest to this problem
has produced many analogous notions and problems concerning
divisors of an integer.
For example, Suryanarayana \cite{Sur} called $N$ to be
superperfect if $\sigma(\sigma(N))=2N$.  It is asked in this paper
and still unsolved whether there were odd superperfect numbers.
Extending the notion of superperfect numbers further,
G. L. Cohen and te Riele \cite{CR} indroduced the notion of
$(m, k)$-perfect numbers, a positive integer $N$ for which
$\sigma(\sigma(\cdots (N)\cdot ))=kN$ with $\sigma$ repeated $m$ times.

Some special classes of divisors have also been studied in several papers.
One of them is the class of unitary divisors defined by Eckford Cohen \cite{CohE}.
A divisor $d$ of $N$ is called a unitary divisor if $(d, N/d)=1$.
Wall \cite{Wal1} introduced the notion of biunitary divisors.
A divisor $d$ of a positive integer $N$ is called a biunitary divisor
if $\gcd(d, n/d)=1$.  For further generalization, see Graeme L. Cohen \cite{CohG}.

According to E. Cohen \cite{CohE} and Wall \cite{Wal1} respecitvely,
$\sigma^*(N)$ and $\sbu(N)$ shall denote the sum of unitary and biunitary divisors of $N$.
Moreover, we write $d\mid\mid N$ if $d$ is a unitary divisor of $N$.  Hence,
for a prime $p$, $p^e\mid\mid N$ if $p$ divides $N$ exactly $e$ times.
Replacing $\sigma$ by $\sigma^*$, Subbarao and Warren \cite{SW}
introduced the notion of a unitary perfect number.  $N$ is called
to be unitary perfect if $\sigma^*(N)=2N$.  They proved
that there are no odd unitary perfect numbers and $6, 60, 90, 87360$
are the first four unitary perfect numbers.  Later the fifth unitary perfect number
has been found by Wall \cite{Wal2}, but no further instance has been found.
Subbarao \cite{Sub} conjectured that there are only finitely many
unitary perfect numbers.

Similarly, a positive integers $N$ is called biunitary perfect if $\sbu(N)=2N$.
Wall \cite{Wal1} showed that $6, 60$ and $90$, the first three unitary perfect numbers,
are the only biunitary perfect numbers.

Combining the notion of superperfect numbers and the notion of unitary divisors,
Sitaramaiah and Subbarao \cite{SS} studied unitary superperfect numbers,
integers $N$ satisfying $\sigma^*(\sigma^*(N))=2N$.  They found all unitary super perfect
numbers below $10^8$ (Further instances are given in A038843 in OEIS \url{https://oeis.org/A038843}).
The first ones are $2, 9, 165, 238$.  Thus
there are both even and odd ones.  The author \cite{Ymd1} showed that
$9, 165$ are all of the odd ones.

Now we can call an integer $N$ satisfying $\sbu(\sbu(N))=2N$
to be {\it biunitary superperfect}.
We can see that $2$ and $9$ are biunitary superperfect as well as unitary superperfect,
while $2$ is also superperfect (in the ordinary sense).

In this paper, we shall determine all biunitary superperfect numbers.

\begin{thm}\label{th1}
$2$ and $9$ are the only biunitary superperfect numbers.
\end{thm}

Theorem \ref{th1} can be thought to be the analogous result for
unitary superperfect numbers by the author \cite{Ymd1}.
Our proof is completely elementary but has some different character
from the proof of the unitary analogue.
Our argument leads to a contradiction that $\sbu(\sbu(N))/N>2$ in many cases,
while Yamada \cite{Ymd1} leads to a contradiction that $\sigma^*(\sigma^*(N))/N<2$.
Moreover, in the biunitary case, we can determine all (odd or even) biunitary superperfect numbers.

Our method does not seem to work to find all odd superperfect numbers.
It prevents us from bounding $\omega(N)$ and $\omega(\sigma(N))$
that $\sigma(p^e)$ with $p$ odd takes an odd value if $e$ is even.
All that we know is the author's result in \cite{Ymd2} that there are only finitely many
odd superperfect numbers $N$ with $\omega(N)\leq k$ or $\omega(\sigma(N))\leq k$ for each {\it given} $k$.

Finally, analogous to G. L. Cohen and te Riele \cite{CR},
we can define a positive integer $N$ to {\it $(m, k)$-biunitary perfect}
if its $m$-th iteration of $\sbu$ equals to $kN$.
We searched for numbers which are $(2, k)$-biunitary perfect for some $k$
(or {\it biunitary multiply superperfect numbers})
and exhaustive search revealed that there exist $173$ integers $N$ below $2^{30}$
dividing $\sbu(\sbu(N))$ including $1$, which are given in Table \ref{tbl}.

Based on our theorem and our search result, we can pose the following problems:
\begin{itemize}
\item For each integer $k\geq 3$, are there infinitely or only finitely many integers $N$ for which $\sbu(\sbu(N))=kN$?
In particular, are the $24$ given integers $N$ all for which $\sbu(\sbu(N))=kN$ with $k\leq 5$?
\item For each integer $k=19$ or $k\geq 21$, does there exist at least one or no integer $N$ for which $\sbu(\sbu(N))=kN$?
\item Are $N=9, 15, 21, 1023, 8925, 15345$ all odd integers diving $\sbu(\sbu(N))$?
\end{itemize}

\begin{table}
\caption{All positive integers $N\leq 2^{30}$ such that $\sbu(\sbu(N))=kN$ for some integer $k$}\label{tbl}
\begin{center}
\begin{small}
\begin{tabular}{| c | c | l |}
 \hline
$k$ & $\# N$'s & $N$ \\
 \hline
$1$ & $1$ & $1$ \\
$2$ & $2$ & $2, 9$ \\
$3$ & $4$ & $8, 10, 21, 512$ \\
$4$ & $8$ & $15, 18, 324, 1023, \ldots, 8925, 15345$ \\
$5$ & $9$ & $24, 30, 144, 288, \ldots, 14976, 23040$ \\
$6$ & $13$ & $42, 60, 160, 270, \ldots, 673254400$ \\
$7$ & $13$ & $240, 1200, 2400, \ldots, 171196416$ \\
$8$ & $18$ & $648, 2808, 3570, \ldots, 1062892908$ \\
$9$ & $26$ & $168, 960, 10368, \ldots, 769600000$ \\
$10$ & $18$ & $480, 2856, 13824, \ldots, 627720192$ \\
$11$ & $8$ & $321408, 1392768, \ldots, 125706240$ \\
$12$ & $26$ & $4320, 10080, \ldots, 779688000$ \\
$13$ & $8$ & $57120, 17821440, \ldots, 942120960$ \\
$14$ & $9$ & $103680, 217728, \ldots, 773760000$ \\
$15$ & $3$ & $1827840, 181059840, 754427520$ \\
$16$ & $4$ & $23591520, \ldots, 594397440$ \\
$17$ & $1$ & $898128000$ \\
$18$ & $1$ & $374250240$ \\
$20$ & $1$ & $11975040$ \\
 \hline
\end{tabular}
\end{small}
\end{center}
\end{table}

\section{Preliminary Lemmas}\label{lemmas}
In this section, we shall give several preliminary lemmas concerning the sum of
infinitary divisors used to prove our main theorems.

Before all, we introduce two basic facts from \cite{Wal1}.
The sum of biunitary divisors function $\sbu$ is multiplicative.
Moreover, if $p$ is prime and $e$ is a positive integer, then
\begin{equation}\label{eq21}
\sbu(p^e)= \begin{cases}
p^e+p^{e-1}+\cdots +1=\frac{p^{e+1}-1}{p-1}, & \text{if $e$ is odd;} \\
\frac{p^{e+1}-1}{p-1}-p^{e/2}=\frac{(p^{e/2}-1)(p^{e/2+1}+1)}{p-1}, & \text{if $e$ is even.}
\end{cases}
\end{equation}
We note that, using the floor function, this can be represented by the single formula:
\begin{equation}\label{eq21x}
\sbu(p^e)=\frac{\left(p^{\floor{\frac{e+2}{2}}}+1\right)\left(p^{\floor{\frac{e+1}{2}}}-1\right)}{p-1}.
\end{equation}

From these facts, we can deduce the following lemmas almost immediately.

\begin{lem}\label{a}
$\sbu(n)$ is odd if and only if $n$ is a power of $2$ (including $1$).
More exactly, $\sbu(n)$ is divisible by $2$ at least $\omega(n)$ times is $n$ is odd
and at least $\omega(n)-1$ times is $n$ is even.
\end{lem}
\begin{proof}
Whether $e$ is even or odd, $\sbu(p^e)$ is odd if and only if $p=2$
by (\ref{eq21}).
Factoring $n=2^e \prod_{i=1}^r p_i^{e_i}$
into distinct odd primes $p_1, p_2, \ldots, p_r$ with $e\geq 0$ and $e_1, e_2, \ldots, e_r>0$,
each $\sbu(p_i^{e_i})$ is even.
Hence $\sbu(n)=\sbu(2^e)\prod_{i=1}^r \sbu(p_i^{e_i})$ is divisible by $2$ at least $r$ times,
where $r=\omega(n)$ if $n$ is odd and $\omega(n)-1$ if $n$ is even.
\end{proof}

\begin{lem}\label{b}
For any prime $p$ and any positive integer $e$, $\sbu(p^e)/p^e\geq 1+1/p^2$.
Moreover, $\sbu(p^e)/p^e\geq 1+1/p$ unless $e=2$
and $\sbu(p^e)/p^e\geq (1+1/p)(1+1/p^3)$ if $e\geq 3$.
More generally, for any positive integers $m$ and $e\geq 2m-1$,
we have $\sbu(p^e)/p^e\geq \sbu(p^{2m})/p^{2m}$
and, unless $e=2m$, $\sbu(p^e)/p^e\geq 1+1/p+\cdots +1/p^m$.
\end{lem}
\begin{proof}
If $e\geq 2m-1$ and $e$ is odd, then $p^e, p^{e-1}, \ldots, p, 1$ are biunitary divisors of $p^e$.
If $e>2m$ and $e$ is even, then $p^e, p^{e-1}, \ldots, p^{e-m}$ are biunitary divisors of $p^e$ since $e-m>e/2$.
Hence, if $e\geq 2m-1$ and $e\neq 2m$, then $\sbu(p^e)=p^e+p^{e-1}+\cdots +1>p^e+\cdots +p^{e-m}=p^e(1+1/p+\cdots +1/p^m)$.
Since $\sbu(p^{2m})/p^{2m}<1+1/p+\cdots +1/p^m$, $\sbu(p^e)/p^e$ with $e\geq 2m-1$ takes its minimum value at $e=2m$.
\end{proof}

Now we shall quote the following lemma of Bang \cite{Ban}, which has been rediscovered
(and extended into numbers of the form $a^n-b^n$)
by many authors such as Zsigmondy\cite{Zsi}, Dickson\cite{Dic} and Kanold\cite{Kan}.
See also Theorem 6.4A.1 in \cite{Sha}.

\begin{lem}\label{lm21}
If $a>b\geq 1$ are coprime integers, then $a^n-1$ has a prime factor
which does not divide $a^m-1$ for any $m<n$, unless $(a, n)=(2, 1), (2, 6)$
or $n=2$ and $a+b$ is a power of $2$.
Furthermore, such a prime factor must be congruent to $1$ modulo $n$.
\end{lem}

As a corollary, we obtain the following lemma:

\begin{lem}\label{c}
Let $p$, $q$ be odd primes and $e$ be a positive integer.
If $\sbu(p^e)=2^a q^b$ for some integers $a$ and $b$, then
a) $e=1$, b) $e=2$ and $p^2+1=2q^b$,
c) $e=3$, $p=2^{a-1}-1$ is a Mersenne prime and $p^2+1=2q^b$ or
d) $e=4$, $p=2^{(a-1)/2}-1$ is a Mersenne prime and $p^2-p+1=q^b$.
Moreover, if $\sbu(2^e)$ is a prime power, then $e\leq 4$.
\end{lem}
\begin{proof}
Let $p$ be an arbitrary prime, which can be $2$.  We set $m=e/2, l=e/2+1$ if $e$ is even
and $m=l=(e+1)/2$ if $e$ is odd.
(\ref{eq21x}) gives that $\sbu(p^e)=(p^l+1)(p^m-1)/(p-1)$ if $e$ is even or odd.

If $m\geq 3$, then, by Lemma \ref{lm21}, $(p^m-1)/(p-1)$ must have a odd prime factor
and $p^l+1$ (if $e$ is even or odd) must have another odd prime factor.
Hence we have $m\leq 2$ and therefore $e\leq 4$.
If $e=1$, then $\sbu(p^e)=\sbu(p)=p+1$, which must be the case a).
If $e=2$, then $\sbu(p^e)=p^2+1$, which must be the case b).
If $e=3$, then $\sbu(p^e)=\sbu(p^3)=(p+1)(p^2+1)$.  If $p$ is odd, then $p+1=2^{a-1}$
and $p^2+1=2q^b$ for some odd prime $q$ since $p^2+1\equiv 2$ (mod $4$).
If $e=4$, then $\sbu(p^e)=(p+1)(p^3+1)=(p+1)^2(p^2-p+1)$, which must be the case d).
\end{proof}

\section{The even case}

Let $N$ be an even biunitary superperfect number.

Firstly, we assume that $\sbu(N)$ is odd.  By Lemma \ref{a}, $N=2^e$ must be a power of $2$.

If $e=2s-1$ is odd and $s>1$, then $\sbu(N)=\sbu(2^{2s-1})=2^{2s}-1=(2^s-1)(2^s+1)$
and $\sbu(\sbu(N))=\sbu(2^s-1)\sbu(2^s+1)\geq 2^s(2^s+2)>2^{2s}$,
which clearly contradicts that $\sbu(\sbu(N))=2N=2^{2s+1}$.

If $e=2s$ is even, then $\sbu(N)=\sbu(2^{2s})=(2^s-1)(2^{s+1}+1)$.
For odd $s>1$, we have $\sbu(\sbu(N))=\sbu(2^s-1)\sbu(2^{s+1}+1)>2^s(2^{s+1}+2)>2^{2s+1}$.
For even $s$, we have $3\mid 2^s-1\mid \sbu(N)$ and therefore
$\sbu(\sbu(N))\geq (10/9)\sbu(N)=(10/9)(2^s-1)(2^{s+1}+1)>2^s(2^{s+1}+1)>2^{2s+1}$.
Hence if $e=2s$ (with $s$ even or odd) and $s>1$, then $\sbu(\sbu(N))>2N$, a contradiction again.

Now we have $e\leq 2$ and we can easily confirm that $2$ is biunitary superperfect but $4$ not.
Hence $N=2$ is the only one in the case $\sbu(N)$ is odd.

Nextly, we assume that $\sbu(N)$ is even and $2^e\mid\mid N, 2^f\mid\mid \sbu(N)$.
We can easily see that
\begin{equation}\label{eq31}
\frac{\sbu(2^f)}{2^f}\cdot\frac{\sbu(2^e)}{2^e}<\frac{\sbu(\sbu(N))}{\sbu(N)}\cdot \frac{\sbu(N)}{N}=2.
\end{equation}

If $e\neq 2$ and $f\neq 2$, then Lemma \ref{b} gives that
$(\sbu(2^f)/2^f)(\sbu(2^e)/2^e)\geq (3/2)^2>2$, which contradicts (\ref{eq31}).
If $e=2$ and $f\geq 3$, then $\sbu(2^f)/2^f\geq 27/16$ and $(\sbu(2^f)/2^f)(\sbu(2^e)/2^e)\geq (27/16)(5/4)>2$,
a contradiction again.  Similarly, we cannot have $e\geq 3$ and $f=2$.

If $(e, f)=(2, 1)$, then $\sbu(2)=3\mid N$ and therefore,
by Lemma \ref{b},
\begin{equation}
\frac{\sbu(\sbu(N))}{N}\geq \frac{10}{9}\cdot\frac{\sbu(2^f)}{2^f}\cdot\frac{\sbu(2^e)}{2^e}=\frac{10}{9}\cdot\frac{15}{8}>2,
\end{equation}
which contradicts the assumption that $\sbu(\sbu(N))=2N$.
Similarly, it is impossible that $(e, f)=(1, 2)$.

The last remaining case is the case $(e, f)=(2, 2)$.
Now we see that $\sbu(2^2)=5$ must divide both $N$ and $\sbu(N)$.
Let $5^g\mid\mid N$ and $5^h\mid\mid \sbu(N)$.  If $g\neq 2$ and $h\neq 2$, then
$\sbu(\sbu(N))/N\geq (5/4)^2(6/5)^2>2$, which is a contradiction again.
If $g\neq 2$ and $h=2$, then $13=(5^2+1)/2$ must divide $N$.
We must have $13^2\mid\mid N$ since otherwise
$\sbu(\sbu(N))/N\geq (5/4)^2(6/5)(26/25)(14/13)>2$, an contradiction.
Since $\sbu(13^2)=2\cdot 5\cdot 17$, $17$ must divide $\sbu(N)$.
Proceeding as above, $17^2$ must divide $\sbu(N)$ and $29=\sbu(17^2)/10$ must divide $N$.
Hence three odd primes $5, 13$ and $29$ must divide $N$
and $2^3$ must divide $\sbu(N)$, which contradicts that $f=2$ in view of Lemma \ref{a}.

Finally, if $g=2$, then $13=\sbu(5^2)/2$ divides both $N$ and $\sbu(N)$.
Let $k$ be the exponent of $13$ dividing $\sbu(N)$.  If any odd prime $p$
other than $5$ divides $\sbu(13^k)$, then three odd primes $5, 13$ and $p$ must divide $N$
and $2^3$ must divide $\sbu(N)$, contradicting that $f=2$ again.
Hence we must have $\sbu(13^k)=2^a 5^b$, which is impossible by Lemma \ref{c}
noting that $\sbu(13)=2\cdot 7$ and $\sbu(13^2)=2\cdot 5\cdot 17$.
Now we have confirmed that $2$ is the only even biunitary superperfect number.

\section{The odd case}

Let $N$ be an odd biunitary superperfect number.
Since $2\mid\mid 2N=\sbu(\sbu(N))$, by Lemma \ref{a}, we have $\sbu(N)=2^f q^g$
and $\sbu(2^f)\sbu(q^g)=2N$ for some odd prime $q$.
Factor $N=\prod_i p_i^{e_i}$ into distinct odd primes $p_i$'s.

Firstly, we consider the case $f=2m-1$ is odd.  Hence $\sbu(2^f)=2^{2m}-1=(2^m-1)(2^m+1)$.

Assume that $m>1$ and take an arbitrary prime factor $p$ of $2^m-1$.
Then $p\leq 2^m-1$ must divide $N$ and therefore
\begin{equation}
\frac{\sbu(\sbu(N))}{N}>\frac{p^2+1}{p^2}\cdot \frac{2^{2m}-1}{2^{2m-1}}
>\frac{2^{2m}}{2^{2m}-1}\cdot \frac{2^{2m}-1}{2^{2m-1}}=2,
\end{equation}
which is impossible.
Hence we must have $m=f=1$ and $\sbu(2^f)=3$ divides $N$.
But, since $\omega(N)\leq m$ by Lemma \ref{a}, we must have $N=3^e$.
By Lemma \ref{c}, we have $e\leq 4$.
Checking each $e$, we see that only $N=3^2$ is appropriate.

Nextly, we consider the case $f=2m$ is even and $\sbu(2^f)=(2^m-1)(2^{m+1}+1)$.

If $2^{m+1}+1$ is composite, then some $p_1\leq \sqrt{2^{m+1}+1}$ must divide $2^{m+1}+1$.
We observe that $2^{m+1}+1=p_1^2$, or equivalently $2^{m+1}=(p_1-1)(p_1+1)$
occurs only when $(m, p_1)=(2, 3)$.
Moreover, it is impossible that $p_1^2+1=2^{m+1}$ since the left cannot be divisible by $4$.
Hence we must have $p_1^2\leq 2^{m+1}-3$ or $(m, p_1)=(2, 3)$.
By the same argument as above, if $p_1^2\leq 2^{m+1}-3$, then we should have
\begin{equation}
\begin{split}
\frac{\sbu(\sbu(N))}{N}> & \frac{p_1^2+1}{p_1^2}\cdot \frac{(2^m-1)(2^{m+1}+1)}{2^{2m}} \\
\geq & \frac{2^{m+1}-2}{2^{m+1}-3}\cdot \frac{(2^m-1)(2^{m+1}+1)}{2^{2m}}=\frac{2^{3m+1}-3\cdot 2^{2m}+1}{2^{3m}-3\cdot 2^{2m-1}} \\
> & 2,
\end{split}
\end{equation}
which is impossible.
If $m=2$ and $p_1=3$, then, since $\sbu(2^4)=3^3$, we must have $e_1=3$ or $e_1=4$
and therefore, by Lemma \ref{b},
\begin{equation}
\frac{\sbu(\sbu(N))}{N}>\frac{\sbu(2^4)}{2^4}\cdot \frac{\sbu(3^{e_1})}{3^{e_1}}
\geq \frac{27}{16}\cdot\frac{112}{81}=\frac{7}{3}>2,
\end{equation}
which is impossible again.

Hence $p_1=2^{m+1}+1$ must be a prime dividing $N$.
By Lemma \ref{c}, we must have $e_1\leq 4$.

If $e_1=1, 3$ or $4$, then $p_1+1=2^{m+1}+2$ divides $\sbu(N)$
and therefore $p_1+1=2(2^m+1)=2 q^l$.  By Lemma \ref{c},
$m=3, 2^3+1=3^2$ or $2^m+1$ must be a prime.
In the latter case, we must have $m=1$ since $2^m+1$ and $p_1=2^{m+1}+1$
are both prime.  Hence $m=1, p_1=5$ or $m=3, p_1=17$ and, in both cases, $q=3$.

The former case $(m, p_1, q)=(1, 5, 3)$ implies that $\sbu(N)=2^2 3^g$ and therefore
$\sbu(5^{e_1})=2^a 3^b$.  Hence we must have $e_1=1$
and $N$ must have the other prime factor $p_2$ such that
$N=5 p_2^{e_2}, \sbu(p_2^{e_2})=2\cdot 3^{g-1}$ and $\sbu(3^g)=2 p_2^{e_2}$.
We see that $e_2=1, p_2=2\cdot 3^{g-1}-1$ and $\sbu(3^g)=2p_2$.
Since $p_2\neq 5$, we must have $g\neq 2$ and
therefore $\sbu(3^g)\geq 4\cdot 3^{g-1}>2p_2$, a contradiction.
Hence we cannot have $(m, p_1, q)=(1, 5, 3)$.
The latter case $(m, p_1, q)=(3, 17, 3)$ implies that $\sbu(N)=2^6 3^g$ and therefore
$\sbu(\sbu(N))/N>(119/64)(10/9)>2=\sbu(\sbu(N))/N$, which is a contradiction again.

Now the remaining is the case $p_1=2^{m+1}+1$ is prime and $e_1=2$, so that $p_1^2+1=2q^l$.
Since $p_1$ must be a Fermat prime, we have $p_1^2+1\equiv 0$ (mod $5$) unless $m=1, p_1=5$.
Hence we must have $p_1=5$ or $p_1>5, p_1^2+1=2\cdot 5^l$.
If $p_1^2+1=2\cdot 5^l$, then St{\o}rmer's result \cite[p.\ 26]{Stm} gives that $p_1=3$ or $7$,
neither of which can occur since $p_1=2^{m+1}+1$ must be a Fermat prime greater than $5$.
Hence the only possibility is that $m=1, p_1=5$ and $q=13$.
We see that $\sbu(N)=2^2 13^g$ and $N$ must have the other prime factor $p_2$ such that $N=5^2 p_2^{e_2}$,
$\sbu(p_2^{e_2})=2\cdot 13^{g-1}$ and $\sbu(13^g)=10p_2^{e_2}$.
By Lemma \ref{c}, we must have $e_2\leq 4$.  However, if $e_2>2$, then
$\sbu(p_2^{e_2})$ must be divisible by $2^2$, which is impossible.
If $e_2=2$, then from St{\o}rmer's result \cite[p.\ 26]{Stm} we obtain that
$p_2=239, \sbu(239^2)=2\cdot 13^4$ and $g=5$, noting that $p_2\neq 5$.
Thus $7=(13+1)/2$ must divide $\sbu(\sbu(N))/2=N=5^2 p_2^{e_2}=5^2\cdot 239^2$, which is absurd.
Finally, if $e_2=1$, then $p_2=2\cdot 13^{g-1}-1$ and $\sbu(13^g)=10p_2>15\cdot 13^{g-1}>\sbu(13^g)$,
which is a contradiction.
Now our proof is complete.

{}
\vskip 12pt

{\small Center for Japanese language and culture, Osaka University,\\ 562-8558, 8-1-1, Aomatanihigashi, Minoo, Osaka, Japan}\\
{\small e-mail: \protect\normalfont\ttfamily{tyamada1093@gmail.com} URL: \url{http://tyamada1093.web.fc2.com/math/}
\end{document}